\documentclass[12pt]{amsart}
\usepackage{amsmath,latexsym,amsfonts,amssymb,amsthm}
\usepackage{geometry}
\usepackage{hyperref}
\usepackage{mathrsfs}
\usepackage{graphicx,color}
\usepackage[alphabetic]{amsrefs}
\usepackage[all,cmtip]{xy}

\newtheorem{prop}{Proposition}[section]
\newtheorem{thm}[prop]{Theorem}
\newtheorem{cor}[prop]{Corollary}

\newtheorem{lem}[prop]{Lemma}

\theoremstyle{definition}

\newtheorem{defn}[prop]{Definition}

\newtheorem{rem}[prop]{\it Remark}

\newtheorem{lemdefn}[prop]{Lemma-Definition}

\newtheorem*{claim*}{Claim}

\newcommand{\bC}{\mathbb{C}}
\newcommand{\bR}{\mathbb{R}}

\newcommand{\bQ}{\mathbb{Q}}

\newcommand{\bN}{\mathbb{N}}

\newcommand{\tS}{\widetilde{S}}
\newcommand{\tD}{\widetilde{D}}

\newcommand{\cO}{\mathcal{O}}

\newcommand{\cF}{\mathcal{F}}
\newcommand{\cJ}{\mathcal{J}}

\newcommand{\fa}{\mathfrak{a}}
\newcommand{\fb}{\mathfrak{b}}
\newcommand{\fc}{\mathfrak{c}}
\newcommand{\fm}{\mathfrak{m}}

\newcommand{\fab}{\fa_{\bullet}}
\newcommand{\fbb}{\fb_{\bullet}}
\newcommand{\fcb}{\fc_{\bullet}}

\newcommand{\rd}{\mathrm{d}}

\newcommand{\Spec}{\mathbf{Spec}}

\newcommand{\mult}{\mathrm{mult}}
\newcommand{\lct}{\mathrm{lct}}

\newcommand{\Aut}{\mathrm{Aut}}
\newcommand{\vol}{\mathrm{vol}}
\newcommand{\ord}{\mathrm{ord}}
\newcommand{\Val}{\mathrm{Val}}

\newcommand{\hvol}{\widehat{\rm vol}}

\numberwithin{equation}{section}

\title{Uniqueness of the minimizer of the normalized volume function}

\author{Chenyang Xu}
\address{Department of Mathematics, MIT, Cambridge, MA, 02139.}
\email{cyxu@mit.edu}
\address{BICMR, Peking University, Beijing, 100871.}
\email{cyxu@math.pku.edu.edu}

\author{Ziquan Zhuang}
\address{Department of Mathematics, MIT, Cambridge, MA, 02139.}
\email{ziquan@mit.edu}

\date{}

\begin{document}

\maketitle

{\let\thefootnote\relax\footnotetext{CX is partially supported by the NSF (No. 1901849).}
\marginpar{}
}

\begin{abstract}
We confirm a conjecture of Chi Li which says that the minimizer of the normalized volume function for a klt singularity is unique up to rescaling. This is achieved by defining stability thresholds for valuations, and then showing that a valuation is a minimizer if and only if it is K-semistable, and that K-semistable valuation is unique up to rescaling. As applications, we prove a finite degree formula for volumes of klt singularities and an effective bound of the local fundamental group of a klt singularity.
\end{abstract}

\section{Introduction}
Throughout this paper, we work over an algebraically closed field $k$ of characteristic 0. Given a klt singularity $x\in (X,\Delta)$, Chi Li introduced in \cite{Li18} the \emph{normalized volume function} $\hvol_{X,\Delta}$ on the space $\Val_{X,x}$ of real valuations on the function field $K(X)$ of $X$ that are centered on $x$. Motivated by the study of K-stability of Fano varieties, the minimizing valuation of $\hvol_{X,\Delta}$ is conjectured to have a number of deep geometric properties, which together comprise the so-called \emph{Stable Degeneration Conjecture}, see \cites{Li18, LX18}.

There has been a lot of progress on the solution of different parts of the Stable Degeneration Conjecture in \cites{Blu18, Li17, LX18, LX20, Xu20}. In particular, it has been known that a minimizing valuation exists (see \cite{Blu18}) and it is always quasi-monomial (see \cite{Xu20}).

\subsection{Main Theorems}

In this paper, we aim to solve another part of the Stable Degeneration Conjecture, namely, the uniqueness of the minimizing valuation, as conjectured in \cite{Li18}*{Conjecture 7.1.2}.
\begin{thm}\label{t-mainunique}
Let $x\in (X,\Delta)$ be a klt singularity, then up to rescaling, there is a unique minimizer $v_0$ of the normalized volume function $\hvol_{X,\Delta}$. 
\end{thm}

We remark that our proof of the theorem does not rely on the fact that the minimizer is quasi-monomial. 

An immediate consequence is the following, which is the local version of the K-semistable case of \cite{Zhu20}*{Theorem 1.1}.

\begin{cor} \label{c-Ginvariant}
If a klt singularity $x\in (X,\Delta)$ admits a group $G$-action, then any minimizer $v_0$ of $\hvol_{X,\Delta}$
is $G$-invariant.
\end{cor}

Another direct consequence is the finite degree formula for normalized volumes. 

\begin{thm}[Finite degree formula]\label{t-finitedegformula}
Let $f\colon \big(y\in (Y,\Delta_Y)\big)\to \big(x\in(X,\Delta)\big)$ be a finite Galois morphism between klt singularities such that $f^*(K_X+\Delta)=K_Y+\Delta_Y$. Then 
$$\hvol(x,X,\Delta)\cdot \deg(f)=\hvol(y,Y,\Delta_Y). $$
\end{thm}

Here $\hvol(x,X,\Delta)$ denotes the volume of the klt singularity $x\in (X,\Delta)$, see Definition \ref{d-normvol}. We apply this to obtain the following effective bound of the local fundamental group. 

\begin{cor}\label{c-localfund}
Let $x\in (X,\Delta)$ be the germ of a klt singularity, then the order of the fundamental group of the smooth locus satisfies
$$\#|\pi_1(x,X^{\rm sm})|\le \frac{n^n}{\hvol(x,X,\Delta)},$$
where the equality holds if and only if $\Delta=0$ and $x\in X$ is \'etale locally isomorphic to $\bC^n/G$ where the action of $G\cong \pi_1(x,X^{\rm sm})$ is fixed point free in codimension one. 
\end{cor}

Combining Corollary \ref{c-localfund} with the results from \cites{Liu18,BJ20} relating local and global volumes of Fano varieties, we also have the following theorem. 

\begin{thm}\label{t-Cartier}
Let $(X,\Delta)$ be a log Fano variety. Then for any $x\in (X,\Delta)$, if we denote by $\pi^{\rm loc}_1(x,X^{\rm sm})$ the local fundamental group of the smooth locus of the germ $x\in (X,\Delta)$,  we have the inequality
$$\#|\pi^{\rm loc}_1(x,X^{\rm sm})|\le \frac{(n+1)^n}{\delta(X,\Delta)^n\cdot \big(-(K_X+\Delta)\big)^n}.$$
In particular, the Cartier index of $X$ is bounded from above by the right hand side of the above inequality.
\end{thm}

Here $\delta(X,\Delta)$ denotes the stability threshold of the log Fano pair $(X,\Delta)$, see \cite{FO18}*{Definition 0.2} or \cite{BJ20}.

\begin{rem}
An interesting application of Theorem \ref{t-Cartier} is that it gives a new proof of the boundedness of K-semistable Fano varieties of a fixed dimension and with volume bounded from below. This was originally proved in \cite{Jia17} as a consequence of the boundedness results proved in \cite{Bir19}. Applying Theorem \ref{t-Cartier}, we only need the fact the Fano varieties with fixed Cartier index form a bounded family, which was first proved in \cite{HMX14}*{Corollary 1.8}.
\end{rem}

\subsection{Outline of the proof}
Given a klt singularity $x\in (X=\Spec(R),\Delta)$, the uniqueness of the minimizer $v$ (up to rescaling) of $\hvol_{X,\Delta}$ is proved in \cite{LX18} {\it under the assumption} that the graded rings associated to the minimizers are finitely generated. The finite generation assumption is used to give a degeneration of the singularity $(X,\Delta)$  to a {\it K-semistable} log Fano cone $(X_0,\Delta_0,\xi_v)$, where $X_0=\Spec({\rm gr}_v(R))$, $\Delta_0$  the degeneration of $\Delta$, and $\xi_v$ is the Reeb vector induced by $v$. This degeneration picture allows one to degenerate any minimizer to $X_0$, and use the strict convexity of the volume function to conclude that $\xi_v$ is the unique $T$-equivariant minimizer on $(X_0,\Delta_0)$ (see \cite{Xu-ICM}*{Page 823}).  

The main aim of this paper is to prove uniqueness of the minimizer without assuming the finite generation property, which still remains a major challenge. For this purpose, a key new input, introduced in Section \ref{ss-ksemi}, is the \emph{K-semistability} of a general valuation $v_0\in \Val_{X,x}$ centered at a klt singularity $x\in (X,\Delta)$. More generally, we will define the stability threshold $\delta(v_0)$ of a valuation $v_0$ with finite log discrepancy. This is done by introducing a local version of basis type divisors. Roughly speaking, a basis type divisor with respect to the chosen valuation $v_0$ is (up to a suitable rescaling factor) a divisor of the form $\{f_1=0\}+\cdots+\{f_N=0\}$ where the images of $f_i$ form a basis of $\cO_{X,x}/\fa_m(v_0)$ (for some integer $m$; here $\fab(v_0)$ denotes the valuation ideals) that is compatible with the filtration induced by $v_0$. Given another valuation $v\in \Val_{X,x}$, we apply the key technical observation from \cite{AZ20} to find basis type divisors that are compatible with both $v_0$ and $v$. This allows us to define the $S$-invariant and $\delta$-invariant of a valuation $v_0$ with respect to another valuation $v$ and to eventually define the local analogue of the stability notions from the global setting. To justify our definition, when $v_0$ is given by a Koll\'ar component $S$, we will show that $\ord_S$ is K-semistable as a valuation if and only if $(S,\Delta_S)$ is K-semistable as a log Fano pair (see Theorem \ref{t-kollarcomponent}). 


With these new definitions, in the second step we show in Section \ref{ss-kimpliesm} that a K-semistable valuation is always a minimizer, and up to scaling there is a unique K-semistable valuation. The observation here is that the log canonical thresholds (lct) of basis type divisors with respect to a K-semistable valuation $v_0$ is asymptotically computed by $v_0$. On the other hand, the asymptotic expected vanishing order of these basis type divisors along a valuation $v$ is at least $\vol_{X,\Delta}(v)^{-1/n}$, with equality when $v=v_0$. Through the identity
\[
\hvol_{X,\Delta}(v)^{1/n}=\frac{A_{X,\Delta}(v)}{\vol_{X,\Delta}(v)^{-1/n}},
\]
minimizing the normalized volume $\hvol_{X,\Delta}(v)$ can be thought of as finding valuations that compute the lct of basis type divisors. In particular, this implies that K-semistable valuations are minimizers of $\hvol_{X,\Delta}$ and the uniqueness then follows from an analysis of the equality condition. 


In the last step, we show that a minimizing valuation $v_0$ is always K-semistable in Section \ref{ss-mimpliesk}. To circumvent the finite generation assumption of ${\rm gr}_{v_0}R$ in \cite{LX18}, we will generalize the derivative argument from \cite{Li17}. Intuitively, given two valuations $v_0,v\in\Val_{X,x}$, we would like to draw a ray between them in the valuation space and use the nonnegativity of the derivative of $\hvol_{X,\Delta}$ at the minimizer $v_0$ to prove its K-semistability. When $v_0$ and $v$ are quasi-monomial with respect to a common stratum, a natural candidate is given by the line joining them in the corresponding dual complex. However, it is unclear to us how to write down such a ray in general. Our idea is to instead construct a family of graded sequences of ideals that interpolates the valuation ideals of the two given valuations. Combining the derivative formula from \cite{Li17} and an analysis of the log canonical thresholds and multiplicities of these ``mixed'' ideal sequences, we can then show that if $v_0$ is a minimizer, then $\delta(v_0)\ge 1$, i.e. $v_0$ is K-semistable. 


\medskip

\noindent {\bf Acknowledgement}: We want to thank Yuchen Liu for discussions, especially for showing us the preprint \cite{Liu20}. We also would like to thank Harold Blum for helpful comments.

\section{Preliminaries}

\noindent{\bf Notation and Conventions:} We follow the notation as in \cites{KM98,Laz-positivity-2,Kol13}.

We say $x\in (X=\Spec(R),\Delta)$ is a singularity if $R$ is a local ring of essentially finite type over $k$, $\Delta$ is an effective divisor on $X$ and $x\in X$ is the unique closed point.

A filtration $\cF^{\bullet}$ on a finite dimensional vector space $V$ is a decreasing sequence $\cF^{t}V$ $(t\in \bR)$ of subspaces satisfying $\cF^{t}V\subseteq \cF^{t'}V$ whenever $t\ge t'$. It is called an $\bN$-filtration if $\cF^0 V=V$ and $\cF^t V=\cF^{\lceil t \rceil} V$ for all $t\in \bR$. For any filtration $\cF$ on $V$, we define its induced $\bN$-filtration $\cF_{\bN}^{\bullet}$ by setting $\cF_{\bN}^{t}V:=\cF^{\lceil t \rceil} V$.

A projective klt pair $(X,\Delta)$ is called a log Fano pair if $-K_X-\Delta$ is ample. 

\subsection{Graded sequence of ideals}

Let $(R,\fm)$ be a local ring of essentially finite type over $k\cong R/\fm$. A graded sequence of ideals (see \cite{JM12}) is a sequence of ideals $\fa_\bullet=(\fa_m)_{m\in \bN}$ such that $\fa_m\cdot \fa_n\subseteq \fa_{m+n}$. We call it \emph{decreasing} if $\fa_{m+1}\subseteq \fa_m$ for all $m\in\bN$. A graded sequence $\fb_{\bullet}$ of ideals is said to be \emph{linearly bounded by} another one $\fa_{\bullet}$, if there is a positive integer $C$ such that
such that
\[
\fb_{Cm}\subset \fa_{m}
\]
for any $m\in\bN$. A finite subset $\{f_1,...,f_N\}$ of $R\setminus \{0\}$ is said to be \emph{compatible with} a decreasing graded sequence $\fab$ of ideals if for all $m\in\bN$, the nonzero images $\bar{f}_i$ of $f_i$ in $R/\fa_m$ are linearly independent. 


The following lemma is a local version of \cite{AZ20}*{Lemma 3.1}.

\begin{lem}\label{l-compatiblebasis}
Let $(R,\fm)$ be a local ring of essentially finite type over $k\cong R/\fm$, let $\fab$ and $\fbb$ be two decreasing graded sequences of $\fm$-primary ideals and let $m\in\bN$. Then there exist some $f_i\in R\setminus \{0\}$ $(1\le i \le N)$ whose images in $R/\fa_m$ form a basis such that $\{f_1,...,f_N\}$ is compatible with both $\fab$ and $\fbb$.
\end{lem}

\begin{proof}
Let $V:=R/\fa_m$ which is a finite dimensional linear space. Then $V$ has two filtrations given by
$$\cF_{\fab}^r V :=(\fa_r+\fa_m)/\fa_m \quad\mbox{and} \quad \cF_{\fbb}^s V:=(\fb_s+\fa_m)/\fa_m.$$
By \cite{AZ20}*{Lemma 3.1}, there exists a basis $\bar{f}_i$ ($1\le i\le N$) of $V$ that is compatible with both filtrations $\cF_{\fab}$ and $\cF_{\fbb}$. We can lift each $\bar{f}_i$ to some element $f_i\in R$ such that $\{f_1,\cdots,f_N\}$ is compatible with $\fbb$ (it suffices to lift each $\bar{f}_i \in \cF_{\fbb}^s V\setminus \cF_{\fbb}^{s+1} V$ to some $f_i\in\fb_s$). On the other hand, since $\bar{f}_i$ is compatible with $\cF_{\fab}$, any such lift is automatically compatible with $\fab$ (i.e. for all $r\le m$, $f_i\in \fa_r$ if and only if $\bar{f}_i\in \cF^r_{\fab} V$).
\end{proof}

\subsection{The space of valuations}

\subsubsection{Valuations}

Let $X$ be a variety defined over $k$. A \emph{real valuation} of its function field
$K(X)$ is a non-constant map $v\colon K(X)^*\to \bR$, satisfying:
\begin{itemize}
 \item $v(fg)=v(f)+v(g)$;
 \item $v(f+g)\geq \min\{v(f),v(g)\}$;
 \item $v(k^*)=0$.
\end{itemize}
We set $v(0)=+\infty$. A valuation $v$ gives rise
to a valuation ring 
$$\cO_v:=\{f\in K(X)\mid v(f)\geq 0\}.$$
We say a valuation $v$ is \emph{centered at} a scheme-theoretic
point $x=c_X(v)\in X$ if we have a local inclusion 
$\cO_{X,x}\hookrightarrow\cO_v$ of local rings.
Notice that the center of a valuation, if exists,
is unique since $X$ is separated. Denote by $\Val_X$ 
the set of real valuations of $K(X)$ that admits a center
on $X$. For a closed point $x\in X$, we further denote by $\Val_{X,x}$ the set
of real valuations of $k(X)$ centered at $x\in X$.  

For each valuation $v\in \Val_{X,x}$ and any positive integer $m$, we define the valuation ideal 
$$\fa_m(v):=\{f\in\cO_{X,x}\mid v(f)\geq m\}.$$
It is clear that $\fa_{\bullet}=\{\fa_m\}_{m\in \bN}$ form a decreasing graded sequence of $\fm_x$-primary ideals.

Let $(X,\Delta)$ be a pair. We denote by
\[
A_{X,\Delta}\colon \Val_X\to \bR \cup\{+\infty\}
\]
the \emph{log discrepancy function of valuations} as in \cite{JM12} and \cite{BdFFU15}*{Theorem 3.1} which extends the standard definition of log discrepancies from divisors to all valuations in $\Val_{X}$. It is possible that $A_{X,\Delta}(v) = +\infty$ for some $v\in \Val_X$, see e.g. \cite{JM12}*{Remark 5.12}. We denote by $\Val^*_X$ the set of valuations $v\in\Val_X$ with $A_{X,\Delta}(v)<+\infty$ and set $\Val^*_{X,x}=\Val^*_X\cap \Val_{X,x}$ for a closed point $x\in X$. Note that $A_{X,\Delta}$ is strictly positive on $\Val_{X}$ if and only if $(X,\Delta)$ is klt.

\begin{prop} \label{p-Izumi}
Let $x\in (X,\Delta)$ be a klt singularity and let $v_0,v_1\in \Val^*_{X,x}$. Then the graded sequences $\fab(v_0)$ and $\fab(v_1)$ of valuation ideals are linearly bounded by each other.
\end{prop}

\begin{proof}
This is a direct consequence of the Izumi type inequalities (see e.g. \cite{Li18}*{Theorem 3.1}), which says that $\fab(v_i)$ and $\{\fm_x^m\}_{m\in\bN}$ are linearly bounded by each other.
\end{proof}

\begin{defn}[Koll\'ar Components]\label{d-kollarcomp}
Let $x\in (X,\Delta)$ be a klt singularity. A prime divisor $S$ over $(X,\Delta)$ is a {\it Koll\'ar component} if there is a birational morphism $\pi\colon Y\to X$ such that $\pi$ is an isomorphism over $X\setminus \{x\}$, $S$ is a prime divisor on $Y$, $\pi(S)=\{x\}$, $-S$ is $\bQ$-Cartier and $\pi$-ample, and $(Y,\pi^{-1}_*\Delta+S)$ is plt. The map $\pi\colon Y\to X$ is called the plt blowup associated to the Koll\'ar component $S$. By adjunction (see \cite{Kol13}*{Definition 4.2}) we may write
$$(K_Y+\pi_*^{-1}\Delta+S)|_S=K_S+\Delta_S,$$
where $(S,\Delta_S)$ is a log Fano pair. 
\end{defn}

\subsubsection{Local volumes}


\begin{defn}\label{d-vol}
Let $X$ be an $n$-dimensional normal variety and let $x\in X$ be a closed point. Following \cite{ELS03} we define the \emph{volume} of a valuation $v\in\Val_{X,x}$ as
\[
\vol(v)=\vol_{X,x}(v)=\limsup_{m\to\infty}\frac{\ell(\cO_{X,x}/\fa_m(v))}{m^n/n!}.
\]
where $\ell(\cdot)$ denotes the length of the Artinian module.
\end{defn}
Thanks to the works of \cites{ELS03, LM09, Cut13}, the above limsup is actually a limit.

\medskip

The following invariant, which was first defined in \cite{Li18}, plays a key role in our study of local stability. 

\begin{defn}[\cite{Li18}]\label{d-normvol}
Let $x\in (X,\Delta)$ be an $n$-dimensional klt singularity.
The \emph{normalized volume function of valuations} $\hvol_{(X,\Delta),x}:\Val_{X,x}\to(0,+\infty)$
is defined as
\[
  \hvol_{(X,\Delta),x}(v)=\begin{cases}
            A_{X,\Delta}(v)^n\cdot\vol_{X,x}(v), & \textrm{ if } v\in \Val^*_{X,x};\\
            +\infty, & \textrm{ if } v\notin  \Val^*_{X,x}.
           \end{cases}
\]
We often denote it by $\hvol_{X,\Delta}$ or $\hvol$ when $x\in (X,\Delta)$ is clear from the context. The \emph{volume of a klt singularity} $(x\in (X,\Delta))$ is defined as
\[
  \hvol(x, X,\Delta):=\inf_{v\in\Val_{X,x}}\hvol_{(X,\Delta),x}(v).
\]

\end{defn}

It has been known that the above infimum is indeed a minimum by \cite{Blu18} and that the minimizing valuations are always quasi-monomial by \cite{Xu20}. The study of $\hvol_{X,\Delta}$ is closely related to K-stability of log Fano pairs, guided by the so-called {\it Stable Degeneration Conjecture} as formulated in \cite{Li18}*{Conjecture 7.1} and \cite{LX18}*{Conjecture 1.2}. See \cite{LLX20} for more background. Our Theorem \ref{t-mainunique} settles one part of this conjecture. 

The following theorem from \cite{LX18} motivates some of our arguments, although we do not need it in our proof.
\begin{thm}
Let $x\in (X=\Spec (R),\Delta)$ be a klt singularity, and $v^{\rm m}$ a minimizer of $\hvol_{X,\Delta}$. Assume the associated grade ring ${\rm gr}_{v^{\rm m}}(R)$ is finitely generated. Denote by $X_0=\Spec({\rm gr}_{v^{\rm m}}(R))$ with the cone vertex $o$, $\Delta_0$ the degeneration of $\Delta$ on $X_0$, $\xi_v$ the Reeb orbit induced by $v$. Then $o\in (X_0,\Delta_0,\xi_v)$ is a K-semistable log Fano cone. 
\end{thm}

Note that the finite generation assumption always holds when $v$ is a divisorial valuation by \cites{LX20, Blu18}.

\subsection{Log canonical thresholds}

\begin{defn}
Given a klt pair $(X,\Delta)$ and a non-zero ideal $\fa$ on $X$, the {\it log canonical threshold} $\lct(X,\Delta;\fa)$ of $\fa$ with respect to $(X,\Delta)$ is defined to be 
\[
    \lct(X,\Delta;\fa) = \max \{t\ge 0\,|\, (X,\Delta+\fa^t) \mbox{ is log canonical}\} = \inf_{v\in\Val^*_X} \frac{A_{X,\Delta}(v)}{v(\fa)}.
\]
\end{defn}

For a graded sequence $\fa_{\bullet}=\{\fa_m\}_{m\in \bN}$ of non-zero ideals on a klt pair $(X,\Delta)$, we can also define its log canonical threshold to be
$$\lct(X,\Delta; \fa_{\bullet}):=\limsup_m m\cdot\lct(X,\Delta; \fa_m)\in \bR_{>0}\cup \{+\infty\}.$$

It is proved in {\cite{Liu18}*{Theorem 27}} that 
\begin{equation} \label{e-hvol=lct^n*mult}
    \hvol(x,X,\Delta)=\inf_{\fa_{\bullet}} \lct(X,\Delta;\fa_{\bullet})^n\cdot \mult(\fa_{\bullet}),
\end{equation}
where the infimum runs through all graded ideal sequences $\fa_{\bullet}$ of $\fm_x$-primary ideals, and $\lct(X,\Delta;\fa_{\bullet})^n\cdot \mult(\fa_{\bullet})$ is set to be $+\infty$ if $\lct(X,\Delta;\fa_{\bullet})=+\infty$.

\section{K-semistability of a valuation}

This is the main section of this paper. We will first define the notion of K-semistability for valuations. Then we will show that for valuations, being K-semistable is the same as being minimizer of the normalized volume function, and such a K-semistable valuation is unique up to rescaling.

\subsection{Definition of K-semistability for a valuation}\label{ss-ksemi}

In this subsection, we introduce a local version of $S$-invariant on the product of the valuation space $\Val^*_{X,x}\times \Val^*_{X,x}$ and use it to define the $\delta$-invariant of valuations, which then naturally give the notion of K-semistability of a valuation. 

Let $x\in (X=\Spec(R),\Delta)$ be a klt singularity.
Fix a valuation $v_0\in \Val^*_{X,x}$. By Proposition \ref{p-Izumi}, for any valuation $v\in \Val^*_{X,x}$, the graded sequences of ideals $\fa_{\bullet}(v)$ and $\fa_{\bullet}(v_0)$ are linearly bounded by each other. By Lemma \ref{l-compatiblebasis}, for any $m\in \bN$ there exist some $f_1,...,f_{N_m}\in R$ (where $N_m=\ell(R/\fa_m(v_0))$) which are compatible with both $\fa_{\bullet}(v_0)$ and $\fa_{\bullet}(v)$ such that their images $\bar{f}_i$ form a basis of $R_m:=R/\fa_m(v_0)$. We call such $\{f_1,...,f_{N_m}\}$ an \emph{$(m,v)$-basis} (\emph{with respect to $v_0$}). The valuation $v$ induces a filtration $\cF_v$ on $R_m$ such that an element $\bar {f}\in R_m$ is contained in  $\cF^\lambda_v R_m$ ($\lambda\in\bR$) if and only if there exists a lifting $f\in R$ of $\bar{f}$ such that $v(f)\ge \lambda$. (For a similar filtration in the global setting, see \cite{BX19}*{5.1.1}).

\begin{lemdefn} \label{l-vol(v_0;v)}
The limit
\[
\vol(v_0;v):=\lim_{m\to \infty} \frac{\ell(\cF^m_v R_m)}{m^n/n!} 
\]
exists. Moreover, we have $\vol(v_0;v/t)=0$ for all $t\gg 0$.
\end{lemdefn}

\begin{proof}
From the definition we have $\cF^m_v R_m=(\fa_m(v)+\fa_m(v_0))/\fa_m(v_0)\cong \fa_m(v)/(\fa_m(v)\cap \fa_m(v_0))$, hence
\[
\ell (\cF^m_v R_m) = \ell(R/(\fa_m(v)\cap \fa_m(v_0)))-\ell(R/\fa_m(v)),
\]
thus by \cite{LM09}*{Theorem 3.8} we obtain
\begin{equation}\label{e-relativevolume}
\lim_{m\to \infty} \frac{1}{m^n/n!} \ell (\cF^m_v R_m )= \mult(\fab(v)\cap\fab(v_0))-\mult(\fab(v)).
\end{equation}

Since $\fab(v_0)$ and $\fab(v)$ are linearly bounded by each other, we have $\fa_{Cm}(v)\subseteq \fa_m(v_0)$ for some constant $C>0$. Thus $\cF_v^{Cm} R_m=0$ and
\[
\vol(v_0;v/t)=\lim_{m\to\infty} \frac{\ell( \cF_v^{tm} R_m )}{m^n/n!}=0
\]
for all $t\ge C$.
\end{proof}

Analogous to the global log Fano case, we set $\tS_m(v_0;v)=\sum^{N_m}_{i=1} \lfloor v(f_i) \rfloor$, which doesn't depend on the choice of $f_i$; indeed it is not hard to check that 
\[
\tS_m(v_0;v)=\sum_{i=0}^{+\infty} i\cdot \ell( \cF_v^i R_m/\cF_v^{i+1} R_m )=\sum_{i=1}^{+\infty} \ell( \cF_v^i R_m ).
\]
We then define
\[
S_m(v_0;v) :=\frac{A_{X,\Delta}(v_0)}{\tS_m(v_0;v_0)}\cdot \tS_m(v_0;v),
\]
\[
S(v_0;v) :=\frac{n+1}{n}\cdot \frac{A_{X,\Delta}(v_0)}{\vol(v_0)} \int_0^\infty \vol(v_0;v/t)\rd t.
\]

\begin{rem}
In the global non-Archimedean setting, a similar construction named the (logarithmic) relative volume of two norms is given in \cite[Section 3]{BJ18}. However, we measure `the relative volume' by taking a quotient instead of a difference.
\end{rem}

\begin{lem} \label{l-Sfunction}
For any $v_0,v\in \Val^*_{X,x}$, we have $S(v_0;v)=\lim_{m\to\infty} S_m(v_0;v)$. Moreover, the function $t\mapsto \vol(v_0;v/t)$ is continuous.
\end{lem}

\begin{proof}
We can embed $(X,\Delta)$ into a projective variety $(\bar{X},\bar{\Delta})$. 
By \cite{LM09}*{Lemma 3.9}, we can find a sufficiently ample line bundle $L$ such that the natural map
\begin{equation} \label{e-ample}
    H^0(\bar{X},L^m)\to H^0(\bar{X},L^m\otimes \cO_X/\fa_{2Cm}(v))
\end{equation}
is surjective for all $m\in\bN$, where $C$ is a positive integer such that $\fa_{Cm}(v)\subseteq\fa_m(v_0)$ and $\fa_{Cm}(v_0)\subseteq\fa_m(v)$ for all $m\in\bN$. Note that this implies that the restriction map
\begin{equation} \label{e-restriction}
    h\colon H^0(\bar{X},L^{m})\to H^0(\bar{X},L^{m}\otimes \cO_X/\fa_{m}(v_0))\cong R/\fa_{m}(v_0)
\end{equation}
is also surjective, where the last isomorphism is given by a trivialization of $L$ near $x$. For such $L$,
\[
W_m:=H^0(\bar{X},L^m\otimes \fa_m(v_0)) \quad\text{and}\quad V_m:=H^0(\bar{X},L^m)
\]
defines two graded linear series $W_\bullet$, $V_\bullet$ that contain ample series. 
The valuation $v$ induces a filtration $\cF_v$ on both $V_\bullet$ and $W_\bullet$ by setting $\cF^\lambda_v V_m = \{s\in H^0(\bar{X},L^m)\,|\,v(s)\ge \lambda\}$ and $\cF^\lambda_v W_m = W_m\cap \cF^\lambda_v V_m$. 

Through \eqref{e-restriction}, the image of $\cF_v$ induces a filtration $\cF_1$ on $R_m:=R/\fa_m(v_0)$. We claim that it is the same as the filtration $\cF^{\bullet}_v$ on $R_m$. Indeed, given an element $f\in \cF_v^\lambda V_m$, it is clear that its image $\bar{f}\in R_m$ lies in $\cF_v^\lambda R_m$. Conversely, if $0\neq\bar{f}\in \cF_v^\lambda R_m$, then it can be lifted to some $f\in R$ with $v(f)\ge \lambda$. Since
$$H^0(\bar{X},L^{m})\to H^0(\bar{X},L^{m}\otimes \cO_X/\fa_{Cm}(v))$$
is a surjective, there exists some $s\in \cF^\lambda_v V_m$ such that $s$ and $f$ has the same image in $R/\fa_{Cm}(v)$. As $\fa_{Cm}(v)\subseteq\fa_m(v_0)$, we see that the restriction of $s$ in $R_m$ gives $\bar{f}$. This proves the claim.

Let $W^t_m=\cF_v^{tm} W_m$ and $V^t_m=\cF_v^{tm} V_m$.
Then from the above claim we have $\ell( \cF_v^{tm} R_m )= \dim V^t_m - \dim W^t_m$, hence
\[
\vol(v_0;v/t) = \vol(V^t_\bullet)-\vol(W^t_\bullet),
\]
which, by \cite{BJ20}*{Proposition 2.3}, is continuous in $t$ when $0\le t\le C$ since $\vol(V^C_\bullet)\ge \vol(W^C_\bullet)>0$ by \eqref{e-ample}; on the other hand, $\vol(v_0;v/t)=0$ when $t\ge C$ as in Lemma \ref{l-vol(v_0;v)}, thus the function $t\mapsto \vol(v_0;v/t)$ is continuous everywhere.

We next prove $S(v_0;v)=\lim_{m\to\infty} S_m(v_0;v)$. We claim that
\begin{equation} \label{e-lim of tS_m}
    \lim_{m\to\infty} \frac{\tS_m(v_0;v)}{m^{n+1}/n!} = \int_0^\infty \vol(v_0;v/t)\rd t.
\end{equation}
By definition, this is equivalent to
\begin{equation} \label{e-lim of tS_m-cont'd}
    \lim_{m\to\infty} \frac{\sum_{i=1}^{\infty} \ell(\cF_v^i R_m)}{m^{n+1}/n!}=\int_0^\infty \vol(v_0;v/t)\rd t.
\end{equation}
Let $\psi_m(t)=\frac{\ell( \cF_v^{\lceil tm \rceil}R_m )}{m^n/n!}$. Then we may rewrite the expression in the above limit as $\int_0^\infty \psi_m(t) \rd t$. Notice that $\lim_{m\to\infty} \psi_m(t)=\vol(v_0;v/t)$ and $\psi_m(t)=0$ for all $t\ge C$ and all $m\in\bN$. The equality \eqref{e-lim of tS_m-cont'd} now follows from the dominated convergence theorem.

It is clear that $\vol(v_0;v_0/t)=\max\{(1-t^n)\vol(v_0),0\}$ for all $t\ge 0$, thus taking $v=v_0$ in \eqref{e-lim of tS_m} we get
\[
\lim_{m\to\infty} \frac{\tS_m(v_0;v_0)}{m^{n+1}/n!} = \int_0^\infty \vol(v_0;v_0/t)\rd t = \int_0^1 (1-t^n)\vol(v_0) \rd t = \frac{n}{n+1}\vol(v_0),
\]
hence
\[
\lim_{m\to\infty} \frac{S_m(v_0;v)}{A_{X,\Delta}(v_0)} = \lim_{m\to\infty} \frac{\tS_m(v_0;v)}{\tS_m(v_0;v_0)} = \frac{n+1}{n}\cdot \frac{\int_0^\infty \vol(v_0;v/t)\rd t}{\vol(v_0)}.
\]
In other words, $S(v_0;v)=\lim_{m\to\infty} S_m(v_0;v)$. 
\end{proof}



\begin{defn}\label{d-ksemi}
A valuation $v_0\in\Val_{X,x}^*$ is said to be {\it K-semistable} if $A_{X,\Delta}(v)\ge S(v_0;v)$ for all $v\in\Val_{X,x}^*$. We also define the stability threshold $\delta(v_0)$ of a valuation $v_0\in\Val_{X,x}^*$ as $\delta(v_0)=\inf_v\delta(v_0;v)$ where $\delta(v_0;v)=\frac{A_{X,\Delta}(v)}{S(v_0;v)}$ and the infimum runs over all valuations $v\in\Val_{X,x}^*$.
\end{defn}

\begin{rem}
The notion of K-semistable valuation has been previously defined for valuations which are quasi-monomial, and whose associated graded rings are finitely generated (see \cite{Xu-ICM}*{Page 819} or \cite{LLX20}*{Theorem 4.14}). Whereas it is known that minimizers of $\hvol_{X,\Delta}$ are quasi-monomial by \cite{Xu20}, the finite generation of the associated graded rings remains open. Therefore, while Definition \ref{d-ksemi} is conjecturally equivalent to the previous definition, we circumvent the issue of finite generation.  
\end{rem}

From the definition it is clear that $\delta(v_0)=\delta(\lambda\cdot v_0)$, thus $v_0$ is K-semistable if and only if $\lambda v_0$ is K-semistable for some $\lambda>0$. In the special case of divisorial valuations induced by Koll\'ar components, we have the following equivalent characterization, which serves as the motivation of our definition.

\begin{thm}\label{t-kollarcomponent}
Let $S$ be a Koll\'ar component over $x\in (X,\Delta)$ (see Definition \ref{d-kollarcomp}).
Then we have $\delta(\ord_S)\ge \min\{1,\delta(S,\Delta_S)\}$ and the valuation $\ord_S$ is K-semistable if and only if the log Fano pair $(S,\Delta_S)$ is K-semistable.
\end{thm}

\begin{proof}

Let $v_0=\ord_S$ and let $v\in\Val^*_{X,x}$. We know there is an ample $\bQ$-divisor $L\sim_{\mathbb{Q}} -S|_S$ on $S$ such that the short exact sequences 
$$0\to \cO_Y(-(m+1)S)\to \cO_Y(-mS)\to \cO_S(mL)\to 0$$ hold (see e.g. \cite{Kol13}*{Section 4.1}), where by convention $\cO_S(mL):=\cO_S(\lfloor mL \rfloor)$. Since $R^1\pi_*\cO_Y(-mS)=0$ for all $m\ge 0$ by Kawamata-Viehweg vanishing, we get isomorphisms 
$$\fa_m/\fa_{m+1}\cong H^0(S,mL)$$ where $\fa_m:=\fa_m(\ord_S)$. After identifying $\oplus_{m\in\bN} \fa_m/\fa_{m+1}$ with $R(S,L):=\oplus_{m\in\bN} H^0(S,mL)$, the valuation $v$ induces a filtration $\cF_v$ on the section ring $R(S,L)$. 

We claim that
\begin{equation} \label{e-S(v0;v)=S(F_v)}
    S(v_0;v)=A_{X,\Delta}(v_0)\cdot S(L;\cF_v),
\end{equation}
where $S(L;\cF_v)$ denotes the $S$-invariant of a filtration as in \cite{BJ20}*{Section 2.5-2.6} (we will also use its approximated versions $S_m(L;\cF_v)$ from \emph{loc. cit.}). To see this, we note that 
\[
\tS_m(v_0;v)=\sum_{i=1}^\infty \sum_{j=1}^m \ell(\cF^i_v (\fa_{j-1}/\fa_j)) = \sum_{j=1}^m j\cdot h^0(S,jL)\cdot S_j(L;(\cF_v)_\bN).
\]
By \cite{BJ20}*{Corollary 2.12}, we have $S_j(L;(\cF_v)_\bN)\to S(L;(\cF_v)_\bN)=S(L;\cF_v)$ ($j\to\infty$), thus as $h^0(S,jL)=(L^{n-1})\frac{j^{n-1}}{(n-1)!}+O(j^{n-2})$, we obtain 
\[
\lim_{m\to\infty} \frac{\tS_m(v_0;v)}{m^{n+1}/(n+1)!}=n(L^{n-1})\cdot S(L;\cF_v).
\]
Thus
\[
\frac{S(v_0;v)}{A_{X,\Delta}(v_0)}=\lim_{m\to\infty} \frac{\tS_m(v_0;v)}{\tS_m(v_0;v_0)} = \frac{S(L;\cF_v)}{S(L;\cF_{v_0})}.
\]
On the other hand, it is clear from the definition that $S(L;\cF_{v_0})=1$ (the filtration $\cF_{v_0}$ satisfies $\cF_{v_0}^j H^0(S,mL)=H^0(S,mL)$ if $j\le m$ and $\cF_{v_0}^j H^0(S,mL)=0$ when $j\ge m+1$), which proves \eqref{e-S(v0;v)=S(F_v)}. 

Since $-(K_S+\Delta_S)\sim_\bQ -(K_Y+\pi^{-1}_*\Delta+S)|_S\sim A_{X,\Delta}(v_0)\cdot L$, we may rewrite \eqref{e-S(v0;v)=S(F_v)} as \begin{equation} \label{e-S(v_0;v)=S(F_v) cont'd}
    S(v_0;v)=S(-(K_S+\Delta_S);\cF_v).
\end{equation}
Let $m\in\bN$ be a sufficiently divisible integer and let $f_1,\cdots,f_N\in \fa_{Am}$ (where $A:=A_{X,\Delta}(v_0)$) be the lift of a basis $\{\bar{f}_i\}$ of $H^0(S,-m(K_S+\Delta_S))=H^0(S,mAL)$. Let $$N:=\dim H^0(S,mAL)\mbox{\ \ \ and\ \ \ }D=\frac{1}{mN}\sum_{i=1}^N \{f_i=0\}.$$ Then we have $\pi^*D=A\cdot S+\tD$ where $\tD|_S$ is an $m$-basis type $\bQ$-divisor of the log Fano pair $(S,\Delta_S)$ (see \cites{FO18,BJ20}). Let $\delta_m:=\min\{1,\delta_m(S,\Delta_S)\}$. From the definition of stability thresholds, we know that the pair $(S,\Delta_S+\delta_m \tD|_S)$ is lc, thus $(Y,S+\pi^{-1}_*\Delta+\delta_m \tD)$ is also lc by inversion of adjunction. We have $$K_Y+S+\pi^{-1}_*\Delta+\delta_m D\ge \pi^*(K_X+\Delta+\delta_m D),$$ hence $(X,\Delta+\delta_m D)$ is lc, which implies that $A_{X,\Delta}(v)\ge \delta_m\cdot v(D)$ for any $v\in\Val^*_{X,x}$ and any $D$ as above. 

If we choose $\bar{f}_i$ to be compatible with the filtration $\cF_v$, then $v(D) = S_m(-(K_S+\Delta_S);\cF_v)$
and we obtain $$A_{X,\Delta}(v)\ge \delta_m\cdot S_m(-(K_S+\Delta_S);\cF_v).$$ Letting $m\to \infty$, we deduce $\delta(v_0)\ge \min\{1,\delta(S,\Delta_S)\}$ using \eqref{e-S(v_0;v)=S(F_v) cont'd}. In particular, if $(S,\Delta_S)$ is K-semistable, then $v_0=\ord_S$ is K-semistable.

Conversely, if $v_0$ is K-semistable, then we have 
$$A_{X,\Delta}(v)\ge S(-(K_S+\Delta_S);\cF_v)$$ 
for any $v\in\Val^*_{X,x}$. Let $c:=v(\fab(v_0))$. We may shift the filtration $\cF_v$ by $c$ to get a new filtration $\cF$ on $R(S,L)$, i.e., $\cF^\lambda H^0(S,mL):=\cF_v^{\lambda+cm} H^0(S,mL)$. It satisfies $\cF^0 H^0(S,mL)=H^0(S,mL)$ as $v(\fa_m)\ge cm$ for all $m\in\bN$. By \cite{BJ20}*{Corollary 2.10}, there exists some $\epsilon_m$ with $\lim_{m\to\infty} \epsilon_m = 1$ such that for all $m\in\bN$ and any $v\in\Val^*_{X,x}$,
\begin{align*}
\epsilon_m\cdot S_m(-(K_S+\Delta_S);\cF)&\le  S(-(K_S+\Delta_S);\cF)\\
&=S(-(K_S+\Delta_S);\cF_v)-A_{X,\Delta}(v_0)\cdot v(\fab(v_0))   \\
&\le  A_{X,\Delta}(v)-A_{X,\Delta}(v_0)\cdot v(\fab(v_0))\ . 
\end{align*}
For sufficiently divisible integer $m$ and with $\{f_i\}$, $D$ and $\tD$ as before, 
this means that $(Y,S+\pi^{-1}_*\Delta+\epsilon_m \tD)$ is lc. By adjunction we see that $(S,\Delta_S+\epsilon_m \tD|_S)$ is lc. Since $\tD|_S$ can be any $m$-basis type $\bQ$-divisor of $(S,\Delta_S)$, we conclude that $\delta_m(S,\Delta_S)\ge \epsilon_m$. Letting $m\to \infty$ we obtain $\delta(S,\Delta_S)\ge 1$, i.e. $(S,\Delta_S)$ is K-semistable.
\end{proof}

In general, if $(S,\Delta_S)$ is not K-semistable, then the inequality in Theorem \ref{t-kollarcomponent} could be strict. 


\subsection{K-semistable valuation is the unique minimizer}\label{ss-kimpliesm}

In this subsection, we show that if $\Val^*_{X,x}$ contains a K-semistable valuation, then it is the unique minimizer of $\hvol_{X,\Delta}$ up to rescaling. 

\begin{thm} \label{t-uniqueness}
Let $x\in(X=\Spec(R),\Delta)$ be a klt singularity and let $v_0\in\Val^*_{X,x}$. Assume that $v_0$ is K-semistable. Then
\begin{enumerate}
    \item $v_0$ is a minimizer of $\hvol_{X,\Delta}$, i.e., $\hvol(x,X,\Delta)=\hvol(v_0)$;
    \item if $v_1\in \Val^*_{X,x}$ is another minimizer of $\hvol_{X,\Delta}$, then $v_1=\lambda v_0$ for some $\lambda>0$.
\end{enumerate}
\end{thm}


For the proof we need some auxiliary calculation. For each valuation $v\in\Val^*_{X,x}$ and every integer $m>0$, we set 
\[
w_m(v):=\min\sum_{i=1}^m \lfloor v(f_i)\rfloor
\]
where the minimum runs over all $f_1,\cdots,f_m\in R\setminus\{0\}$ that are compatible with $\fab(v)$. Clearly the minimum is achieved by some $f_1,\cdots,f_m$ that are compatible with $\fab(v)$, if and only if for the unique integer $r$ satisfying $\ell(R/\fa_{r+1}(v)) >m\ge \ell(R/\fa_r(v))$, $f_1,\cdots,f_m$ span $R/\fa_r(v)$ and form a linearly independent set in $R/\fa_{r+1}(v)$.

\begin{lem} \label{l-asymptotic of w_m}
We have 
\[
\lim_{m\to\infty}  \frac{w_m(v)}{m^{\frac{n+1}{n}} }=\frac{n}{n+1}\cdot \left(\frac{n!}{\vol(v)}\right)^{1/n}.
\]
\end{lem}

\begin{proof}
Let $\fab=\fab(v)$. From the above description we have
\[
0\le w_m(v)-\sum_{i=0}^{r-1} i\cdot \ell(\fa_i/\fa_{i+1})\le r\cdot \ell(\fa_r/\fa_{r+1})
\]
for all integers $r,m>0$ with $\ell(R/\fa_r)\le m <\ell(R/\fa_{r+1})$. Note that this implies
\[
\lim_{r\to\infty} \frac{m}{r^n/n!}=\vol(v).
\]
We also have $\lim_{r\to\infty} \frac{\ell(\fa_r/\fa_{r+1})}{r^n}=0$ and
\begin{align*}
    \lim_{r\to\infty} \frac{\sum_{i=0}^r i\cdot \ell(\fa_i/\fa_{i+1})}{r^{n+1}/n!} & = \lim_{r\to\infty} \frac{r\cdot \ell(R/\fa_{r+1})-\sum_{i=0}^r \ell(R/\fa_i)}{r^{n+1}/n!} \\
    & =\left(1-\frac{1}{n+1}\right)\vol(v).
\end{align*}
Thus $\lim_{r\to\infty} \frac{w_m(v)}{r^{n+1}/n!}=\frac{n}{n+1}\vol(v)$ and
\[
\lim_{m\to\infty} \frac{w_m(v)}{m^{\frac{n+1}{n}}} = \lim_{r\to\infty} \left( \frac{w_m(v)}{r^{n+1}/n!}\cdot \frac{r^n/n!}{m}\cdot \frac{r}{m^{1/n}}\right) = \frac{n}{n+1}\cdot \left(\frac{n!}{\vol(v)}\right)^{1/n}.
\]
\end{proof}

\begin{proof}[Proof of Theorem \ref{t-uniqueness}]
We first prove that $v_0$ is a minimizer of $\hvol_{X,\Delta}$, i.e. $\hvol(v)\ge \hvol(v_0)$ for every valuation $v\in\Val^*_{X,x}$. Without loss of generality we may assume that $A_{X,\Delta}(v_0)=A_{X,\Delta}(v)=1$. Let $m\in\bN$ and let $f_1,\cdots,f_{N_m}$ be an $(m,v)$-basis with respect to $v_0$ (where $N_m=\ell(R/\fa_m(v_0))$).
Since $v_0$ is K-semistable, we have
\begin{equation} \label{e-A>=delta*v(D)}
    1=A_{X,\Delta}(v)\ge S(v_0;v).
\end{equation}
From the definition it is clear that $\tS_m(v_0;v_0)=w_{N_m}(v_0)$ and $\tS_m(v_0;v)\ge w_{N_m}(v)$, 
hence by Lemma \ref{l-asymptotic of w_m} we get
\[
S(v_0;v)\ge \lim_{m\to\infty} \frac{w_{N_m}(v)}{w_{N_m}(v_0)} = \left(\frac{\vol(v_0)}{\vol(v)}\right)^\frac{1}{n}.
\]
Combined with \eqref{e-A>=delta*v(D)} we immediately have 
\[
\hvol(v)=\vol(v)\ge \vol(v_0)=\hvol(v_0),
\]
i.e. $v_0$ minimizes the normalized volume function $\hvol_{X,\Delta}$.

Now assume $\vol(v_0)=\vol(v)$. We claim that
\begin{equation} \label{e-vol(v)=vol(v_0)}
    \vol(v_0)=\vol(v)=\mult(\fab(v_0)\cap\fab(v)).
\end{equation}
Suppose this is not the case, then $\vol(v_0;v)>0$ by \eqref{e-relativevolume}. Thus by the continuity part of Lemma \ref{l-Sfunction}, there exists some $\epsilon>0$ such that
\[
\gamma:=\vol\left(v_0;\frac{v}{1+2\epsilon}\right)=\lim_{m\to\infty} \frac{\ell( \cF_v^{(1+2\epsilon)m}(R/\fa_m(v_0)) )}{m^n/n!} > 0.
\]
For each $m\in\bN$, let $k_m$ be the unique integer $k$ determined by
\[
\ell(R/\fa_{k-1}(v))\le N_m<\ell(R/\fa_k(v)).
\]
Since $\vol(v_0)=\vol(v)$, we have $\lim_{m\to\infty} \frac{k_m}{m}=1$ and thus $k_m<(1+\epsilon)m$ for sufficiently large $m$. Let $g_1,\cdots,g_{N_m}\in R\setminus\{0\}$ be a sequence that's compatible with $\fab(v)$ such that 
\[
w_{N_m}(v)=\sum_{i=1}^{N_m} \lfloor v(g_i) \rfloor.
\]
Then by construction we have $v(g_i)\le k_m$ for all $1\le i\le N_m$ and the inequality
\[
\tS_m(v_0;v)=\sum_{i=1}^{N_m} \lfloor v(f_i) \rfloor \ge w_{N_m}(v)
\]
can be upgraded as
\[
\sum_{i=1}^{N_m} \min\{\lfloor v(f_i) \rfloor ,k_m\} \ge \sum_{i=1}^{N_m} \lfloor v(g_i)\rfloor = w_{N_m}(v).
\]
In particular, for sufficiently large $m$ we get
\begin{align*}
    \sum_{i=1}^{N_m} \lfloor v(f_i) \rfloor & = \sum_{j=0}^\infty j \cdot \ell( \cF_v^j R_m/\cF_v^{j+1} R_m )\\
    & \ge \sum_{j=0}^\infty \min\{j,k_m\} \cdot \ell( \cF_v^j R_m/\cF_v^{j+1} R_m )+ ((1+2\epsilon)m-k_m) \cdot \ell( \cF_v^{(1+2\epsilon)m}R_m )\\
    & \ge \sum_{j=0}^\infty \min\{j,k_m\} \cdot \ell( \cF_v^j R_m/\cF_v^{j+1} R_m )+ \epsilon m \cdot \frac{\gamma m^n}{n!} \\
    & = \sum_{i=1}^{N_m} \min\{\lfloor v(f_i)\rfloor,k_m\} + \frac{\epsilon\gamma m^{n+1}}{n!} \\
    & \ge w_{N_m}(v) + \frac{\epsilon\gamma m^{n+1}}{n!},
\end{align*}
where $R_m=R/\fa_m(v_0)$. Dividing by $\sum \lfloor v_0(f_i) \rfloor=\tS_m(v_0;v_0)=w_{N_m}(v_0)=O(m^{n+1})$ and letting $m\to\infty$, we obtain 
\[
1\ge S(v_0;v) = \lim_{m\to\infty} \frac{\tS_m(v_0;v)}{\tS_m(v_0;v_0)} > \lim_{m\to\infty} \frac{w_{N_m}(v)}{w_{N_m}(v_0)} = \left(\frac{\vol(v_0)}{\vol(v)}\right)^\frac{1}{n}
\]
where the last equality follows from Lemma \ref{l-asymptotic of w_m}, hence $\vol(v)>\vol(v_0)$, a contradiction. This proves the claim \eqref{e-vol(v)=vol(v_0)}. By the following Lemma \ref{l-valuation equal}, it implies $v=v_0$ and we are done.
\end{proof}

The following result, which is an improvement of \cite{LX20}*{Proposition 2.7}, is used in the above proof.

\begin{lem} \label{l-valuation equal}
Let $x\in X=\Spec(R)$ be a singularity and let $v_0,v_1\in\Val^*_{X,x}$. Assume that
\[
\vol(v_0)=\vol(v_1)=\mult(\fab(v_0)\cap\fab(v_1))>0.
\]
Then $v_0=v_1$.
\end{lem}

\begin{proof}
We prove by contradiction. Assume that $v_0(f)\neq v_1(f)$ for some $f\in R$. Without loss of generality we may assume that $v_0(f)=\ell_0>\ell_1=v_1(f)$. Replacing $f$ by $f^k$ for some $k\in\bN$ we may further assume that $\ell_0 \ge \ell_1+1$. For $v\in\Val^*_{X,x}$ and $r\ge 0$, let $\fa_r(v)=\{f\in R\,|\,v(f)\ge r\}$. Let $\fb_r=\fa_r(v_0)\cap \fa_r(v_1)$ and $\fc_r=\fa_r(v_0)\cap \fa_{2r}(v_1)$ where $r\ge 0$.

For every $m\in\bN$ and every $s\in \fb_m$, we have $v_0(f^m s) = m\cdot v_0(f)+v_0(s)\ge m(\ell_0+1)$, thus multiplication by $f^m$ induces a map
\[
\fb_m\stackrel{f^m}{\longrightarrow} \fa_{m(\ell_0+1)}(v_0)\to \fa_{m(\ell_0+1)}(v_0)/\fb_{m(\ell_0+1)}
\]
whose kernel is contained in $\fc_m$ (since $v_1(f^m s)\ge m(\ell_0+1)$ implies $v_1(s)\ge m(\ell_0+1)-m\ell_1\ge 2m$). It follows that
\begin{equation} \label{e-dim(b_m/c_m)}
    \ell(\fa_{m(\ell_0+1)}(v_0)/\fb_{m(\ell_0+1)})\ge \ell(\fb_m/\fc_m)
\end{equation}
for all $m\in\bN$. 
By \cite{LX20}*{Proposition 2.7}, there exists some $0\neq g\in\fm_x$ such that $\ell_2=v_1(g)> v_0(g)>0$. For every $m\in\bN$ and every $s\in \fc_m$, we then have $$v_1(g^m s) = m\cdot v_1(g)+v_1(s)\ge m(\ell_2+2),$$ thus multiplication by $g^m$ induces a map
\[
\fc_m\stackrel{g^m}{\longrightarrow} \fa_{m(\ell_2+2)}(v_1)\to \fa_{m(\ell_2+2)}(v_1)/\fb_{m(\ell_2+2)}
\]
whose kernel is contained in $\fb_{2m}$ (if $v_0(g^m s)\ge m(\ell_2+2)$ then as $v_0(g)\le \ell_2$ we get $v_0(s)\ge 2m$). It follows that \begin{equation} \label{e-dim(c_m/b_2m)}
    \ell(\fa_{m(\ell_2+2)}(v_1)/\fb_{m(\ell_2+2)})\ge \ell(\fc_m/\fb_{2m})
\end{equation}
for all $m\in\bN$. Combining \eqref{e-dim(b_m/c_m)} and \eqref{e-dim(c_m/b_2m)} we see that
\begin{align*}
    & (\ell_0+1)^n (\mult(\fbb)-\vol(v_0)) + (\ell_2+2)^n (\mult(\fbb)-\vol(v_1)) \\
    =& \lim_{m\to\infty} \frac{\ell(\fa_{m(\ell_0+1)}(v_0)/\fb_{m(\ell_0+1)})}{m^n/n!} + \lim_{m\to\infty} \frac{\ell(\fa_{m(\ell_2+2)}(v_1)/\fb_{m(\ell_2+2)})}{m^n/n!}\\
    \ge & \lim_{m\to\infty} \frac{\ell(\fb_m/\fc_m)+\ell(\fc_m/\fb_{2m})}{m^n/n!}
    =\lim_{m\to\infty} \frac{\ell(\fb_m/\fb_{2m})}{m^n/n!} \\
    = & (2^n-1)\mult(\fbb)>0,
\end{align*}
which contradicts our assumption. Thus $v_0(f)=v_1(f)$ for all $f\in R$ as desired.
\end{proof}

\subsection{Every minimizer is K-semistable}\label{ss-mimpliesk}

In this subsection, we show that every valuation that minimizes the normalized volume function is K-semistable. Combined with Theorem \ref{t-uniqueness}, this proves the uniqueness of the minimizer.

\begin{thm} \label{t-minimizer K-ss}
Let $x\in(X=\Spec(R),\Delta)$ be a klt singularity and let $v_0\in\Val^*_{X,x}$ be a minimizer of the normalized volume function $\hvol_{X,\Delta}$. Then $v_0$ is K-semistable. 
\end{thm}

In other words, we will show that $A_{X,\Delta}(v)\ge S(v_0;v)$ for every valuation $v\in\Val^*_{X,x}$. Inspired by the argument of \cite{Li17}, we consider a family $\fb_{\bullet,t}$ ($t\in\bR_{\ge 0}$) of graded sequences of ideals that interpolate the valuation ideal sequences of $v_0$ and $v$, defined as follows: we set $\fb_{\bullet,0}=\fab(v_0)$; when $t>0$, we set
\begin{equation} \label{e-b_m,t}
    \fb_{m,t}=\sum_{i=0}^m \fa_{m-i}(v_0)\cap \fa_i(tv).
\end{equation}
Roughly speaking, the ideal $\fb_{m,t}$ is generated by elements $f\in R$ with $v_0(f)+t\cdot v(f)\ge m$. By \eqref{e-hvol=lct^n*mult}, we have 
\[
\lct(\fb_{\bullet,t})^n\cdot \mult(\fb_{\bullet,t})\ge \hvol(v_0)=\lct(\fb_{\bullet,0})^n\cdot \mult(\fb_{\bullet,0}).
\]
To relate this to the K-semistability of $v_0$, the idea is to take the derivative of the above normalized multiplicities at $t=0$, which was a technique introduced in \cite{Li17}. To do so we analyze the log canonical thresholds and multiplicities of $\fb_{\bullet,t}$.

\subsubsection{Log canonical threshold of summation}
We first establish an inequality for the log canonical thresholds of graded sequences of ideals. Given two graded sequences of ideals $\fa_{\bullet}$ and $\fb_{\bullet}$, we define $\fcb:=\fab\boxplus\fbb$ by setting
\[
\fc_m=(\fa\boxplus\fb)_m=\sum_{i=0}^m \fa_i\cap\fb_{m-i}.
\]
It is easy to verify that $\fcb$ is also a graded sequence of ideals. Note that our definition differs from the usual sum of ideal sequences (see e.g. \cite{Mus02}) since we use intersections of ideals rather than taking product.

\begin{thm}\label{t-lctsum}
Under the above notation, assume $\fa_{\bullet}$ and $\fb_{\bullet}$ are graded sequences of $\fm_x$-primary ideals. Then we have 
$$\lct(\fc_{\bullet})\le \lct(\fa_{\bullet})+\lct(\fb_{\bullet}).$$
\end{thm}

We denote by $\cJ(\fa^t)$ the multiplier ideal of a fractional ideal and similarly by $\cJ(\fab^t)$ the asymptotic multiplier ideal of a graded sequence of ideals $\fab$ with exponent $t$ (see \cite{Laz-positivity-2} for details). The above inequality will follow from a summation formula of multiplier ideals.

\begin{lem}\label{l-summation}
For any two graded sequences of ideals $\fa_{\bullet}$, $\fb_{\bullet}$ and any $t>0$, we have
\begin{equation} \label{e-summation}
    \cJ(\fcb^t)\subseteq \sum_{\lambda+\mu=t} \cJ(\fab^\lambda) \cap \cJ(\fbb^\mu)
\end{equation}
where $\fc_{\bullet}=\fa_{\bullet}\boxplus\fb_{\bullet}$.
\end{lem}

\begin{proof}
We follow the proof of \cite{Takagi-multiplier-ideal}*{Proposition 4.10}. Let $m$ be a sufficiently large and divisible integer such that $\cJ(\fcb^t)=\cJ(\fc_m^{t/m})$. By the summation formula of multiplier ideals (see \cite{Takagi-multiplier-ideal}*{Theorem 0.1(2)}), which says that for any two ideals $\fa$ and $\fb$,
$$\cJ\big((\fa+\fb)^t\big)=\sum_{t_1+t_2=t}\cJ(\fa^{t_1}\cdot \fb^{t_2}),$$
we have
\[
\cJ(\fc_m^{t/m})=\cJ\left(\left(\sum_{i=0}^m \fa_i\cap\fb_{m-i}\right)^{t/m}\right)=\sum_{t_0+\cdots+t_m=t/m} \cJ\left(\prod_{i=0}^m (\fa_i\cap\fb_{m-i})^{t_i}\right).
\]
(The right hand side is a finite sum.) Since $\fa_i^{m!/i}\subseteq \fa_{m!}$, each individual term in the above right hand side is contained in
\[
\cJ\left(\prod_{i=0}^m \fa_i^{t_i}\right)\subseteq \cJ\left(\prod_{i=0}^m \fa_{m!}^{\frac{it_i}{m!}}\right)= \cJ\left(\fa_{m!}^{\lambda/m!}\right)\subseteq \cJ(\fab^\lambda)
\]
where $\lambda=\sum_{i=0}^m it_i$. By symmetry, it is also contained in $\cJ(\fbb^\mu)$ where $\mu=\sum_{i=0}^m (m-i)t_i$. Note that $\lambda+\mu=\sum_{i=0}^m mt_i=m\cdot \frac{t}{m}=t$, thus every
\[
\cJ\left(\prod_{i=0}^m (\fa_i\cap\fb_{m-i})^{t_i}\right)\subseteq \cJ(\fab^\lambda) \cap \cJ(\fbb^\mu)
\]
is contained in the right hand side of \eqref{e-summation}. This completes the proof.
\end{proof}

\begin{proof}[Proof of Theorem \ref{t-lctsum}]
Let $\alpha=\lct(\fa_{\bullet})$, $\beta=\lct(\fb_{\bullet})$ and let $t=\alpha+\beta$. For any $\lambda,\mu\ge 0$ with $\lambda+\mu=t$ we have either $\lambda\ge \alpha$ or $\mu\ge \beta$, therefore $\cJ(\fab^\lambda) \cap \cJ(\fbb^\mu)\subseteq \fm_x$. By Lemma \ref{l-summation} we see that $\cJ(\fc_{\bullet}^t)\subseteq \fm_x$ and hence $\lct(\fc_{\bullet})\le t= \lct(\fa_{\bullet})+\lct(\fb_{\bullet})$.
\end{proof}

\subsubsection{Multiplicities of a family of graded sequences of ideals}

We next derive a formula for the multiplicities of $\fb_{\bullet,t}$.

\begin{lem} \label{l-multiplicity}
$\mult(\fb_{\bullet,t})=\vol(v_0)-(n+1) \int_0^\infty \vol(v_0;v/u)\frac{t\rd u}{(1+tu)^{n+2}}$.
\end{lem}

\begin{proof}
By definition, we have
\[
\mult(\fb_{\bullet,t})=\lim_{m\to\infty} \frac{\ell (R/\fb_{m,t})}{m^n/n!}.
\]
However, to derive the statement of the lemma, it is better to use a different formula, which follows from the above equality:
\begin{equation} \label{e-mult(b_t)-1}
    \mult(\fb_{\bullet,t})=\lim_{m\to\infty} \frac{\sum_{j=1}^m \ell (R/\fb_{j,t})}{m^{n+1}/(n+1)!}.
\end{equation}
For ease of notation, let $\fab=\fab(v_0)$. We have 
\[
\big(\fa_{j-\ell-1}\cap \fa_{\ell+1}(tv)\big)\cap \sum_{i=0}^\ell \big(\fa_{j-i}\cap \fa_i(tv)\big) = \fa_{j-\ell}\cap \fa_{\ell+1}(tv)
\]
for all $0\le \ell<j$ and we get short exact sequences
\[
0\to \frac{\fa_{j-\ell-1}\cap \fa_{\ell+1}(tv)}{\fa_{j-\ell}\cap \fa_{\ell+1}(tv)} \to \frac{R}{\sum_{i=0}^\ell \fa_{j-i}\cap \fa_i(tv)}\to \frac{R}{\sum_{i=0}^{\ell+1} \fa_{j-i}\cap \fa_i(tv)}\to 0.
\]
Thus from the definition of $\fb_{\bullet,t}$, we get
\[
\ell(R/\fb_{j,t})=\ell(R/\fa_j)-\sum_{i=1}^{j} \ell( \cF^{i/t}_v (\fa_{j-i}/\fa_{j-i+1}) ).
\]
Summing over $j=0,1,\cdots,m$ we obtain
\begin{align*}
    \sum_{j=0}^m \ell (R/\fb_{j,t}) &= \sum_{j=1}^m \ell(R/\fa_j) - \sum_{1\le i\le j\le m} \ell( \cF^{i/t}_v (\fa_{j-i}/\fa_{j-i+1}) )\\
    &= \sum_{j=1}^m \ell(R/\fa_j) - \sum_{i=1}^m \ell( \cF^{i/t}_v (R/\fa_{m-i+1}) )
\end{align*}
Combining with \eqref{e-mult(b_t)-1}, we deduce that
\begin{equation} \label{e-mult(b_t)-2}
    \mult(\fb_{\bullet,t})=\vol(v_0) - (n+1)\cdot \lim_{m\to\infty} \frac{W_m}{m^{n+1}/n!}
\end{equation}
where $W_m:=\sum_{i=1}^m \ell( \cF^{i/t}_v (R/\fa_{m-i+1}) )$. To analyze the limit in the above expression, we set (c.f. the proof of Lemma \ref{l-Sfunction} or the argument in \cite{Li17}*{Section 4.1.1})
\begin{align*}
    \phi_m(y) & =\frac{\ell( \cF^{\lceil my\rceil/t}_v (R/\fa_{m-\lceil my\rceil + 1}) )}{m^n/n!}\\
    & =\frac{\ell( \cF^{\lceil my\rceil/t}_v (R/\fa_{m-\lceil my\rceil + 1}) )}{(m-\lceil my\rceil + 1)^n/n!}\cdot \frac{(m-\lceil my\rceil + 1)^n}{m^n}
\end{align*}
where $0< y< 1$. It is not hard to check that 
\[
\lim_{m\to\infty} \phi_m(y)=g\left(\frac{y}{t(1-y)}\right)(1-y)^n
\]
where $g(u)=\vol(v_0;v/u)$, hence by the dominated convergence theorem we have
\begin{align*}
    \lim_{m\to\infty} \frac{W_m}{m^{n+1}/n!} & = \lim_{m\to\infty}\int_0^1 \phi_m(y) \rd y \\
    & = \int_0^1 g\left(\frac{y}{t(1-y)}\right)(1-y)^n \rd y = \int_0^\infty g(u)\frac{t\rd u}{(1+tu)^{n+2}}.
\end{align*}
Together with \eqref{e-mult(b_t)-2} this implies the statement of the lemma.
\end{proof}

We are now ready to give the proof of Theorem \ref{t-minimizer K-ss}.

\begin{proof}[Proof of Theorem \ref{t-minimizer K-ss}]
Let $v\in\Val^*_{X,x}$. Up to rescaling, we may assume that $A_{X,\Delta}(v_0)=A_{X,\Delta}(v)=1$. Define $\fb_{\bullet,t}$ ($t\ge0$) as in \eqref{e-b_m,t}, and let 
$$f(t) :=(1+t)^n\cdot\mult(\fb_{\bullet,t}).$$ Clearly $f(0)=\hvol(v_0)$.
By Theorem \ref{t-lctsum} we have
\[
\lct(\fb_{\bullet,t})\le \lct(\fab(v_0))+\lct(\fab(tv))\le \frac{A_{X,\Delta}(v_0)}{v_0(\fab(v_0))}+\frac{A_{X,\Delta}(v)}{v(\fab(tv))}\le 1+t.
\]
Hence for all $t\ge 0$,
$$f(t)\ge \lct(\fb_{\bullet,t})^n\cdot \mult(\fb_{\bullet,t})\ge \hvol(v_0)=f(0),$$ where the second inequality follows from \eqref{e-hvol=lct^n*mult} and the assumption that  $v_0$ is a minimizer of $\hvol_{X,\Delta}$. Thus $f'(0)\ge 0$. Using Lemma \ref{l-multiplicity}, we find
\[
f'(0)=n\cdot \vol(v_0)-(n+1)\int_0^\infty \vol(v_0;v/u)\rd u,
\]
thus 
\[
A_{X,\Delta}(v)=1 \ge \frac{n+1}{n}\cdot \frac{\int_0^\infty \vol(v_0;v/u)\rd u}{\vol(v_0)} = S(v_0;v).
\]
Since $v\in\Val^*_{X,x}$ is arbitrary, it follows that $v_0$ is K-semistable.
\end{proof}

\begin{rem}
If we combine together Theorems \ref{t-kollarcomponent}, \ref{t-uniqueness} and \ref{t-minimizer K-ss}, we get a proof of the fact that a Koll\'ar component is a minimizer if and only if it is K-semistable, which was first established in \cites{Li17, LX20}. While in the proof of Theorem \ref{t-minimizer K-ss}, we still use a version of the derivative formula introduced in \cite{Li17}, we do not need it for the converse.
\end{rem}

\section{Applications}

In this section, we prove the results mentioned in the introduction.

\begin{proof}[Proof of Theorem \ref{t-mainunique}]
By \cite{Blu18} (see also \cite{Xu20}*{Remark 3.8}), there exists $v_0\in\Val^*_{X,x}$ such that $\hvol(v_0)=\hvol(x,X,\Delta)$. By Theorem \ref{t-minimizer K-ss}, $v_0$ is K-semistable, thus by Theorem \ref{t-uniqueness}, it is the unique minimizer of the normalized volume function up to scaling.
\end{proof}

\begin{proof}[Proof of Corollary \ref{c-Ginvariant}]
If $v_0$ is a minimizer of $\hvol_{X,\Delta}$, then for any $g\in G$, the valuation $g\cdot v_0$ defined by $(g\cdot v_0)(s)=v_0(g^{-1}\cdot s)$ is also a minimizer of $\hvol_{X,\Delta}$. By Theorem \ref{t-mainunique}, we have $g\cdot v_0=\lambda v_0$ for some $\lambda>0$; but since $A_{X,\Delta}(v_0)=A_{X,\Delta}(g\cdot v_0)$, we must have $\lambda=1$, hence $v_0=g\cdot v_0$ is $G$-invariant.
\end{proof}

Theorem \ref{t-finitedegformula} is then an easy consequence of Corollary \ref{c-Ginvariant}. 

\begin{proof}[Proof of Theorem \ref{t-finitedegformula}]
Let $G=\Aut(Y/X)$ be the Galois group. By Corollary \ref{c-Ginvariant}, the minimizer $v_0$ of $\hvol_{Y,\Delta_Y}$ is $G$-invariant, hence $\hvol(y,Y,\Delta_Y)=\hvol^G(y,Y,\Delta_Y)$, where 
\[
\hvol^G(x,X,\Delta):=\inf_{v\in\Val^G_{X,x}}\hvol_{(X,\Delta),x}(v)
\]
as the infimum runs over all valuations $v\in\Val_{X,x}$ that are invariant under the $G$-action. 

By \cite{LX-cubic-3fold}*{Theorem 2.7(1)}, we get $\hvol^G(y,Y,\Delta_Y)=|G|\cdot \hvol(x,X,\Delta)$ (in \emph{loc. cit.} it is assumed that $\Delta_Y=0$ and $f$ is \'etale in codimension one, but the proof applies in general since these assumptions are only used to guarantee that $\Delta=0$). Thus 
\[
\hvol(y,Y,\Delta_Y)=|G|\cdot \hvol(x,X,\Delta)=\deg(f)\cdot \hvol(x,X,\Delta).
\]
\end{proof}

In fact, the above argument implies the finite degree formula for any quasi-\'etale (i.e. \'etale in codimension one) finite morphism $Y\to X$, as we can pass to the Galois closure of $Y/X$, which is also quasi-\'etale. However, if there is a branched divisor, then the pull back of $K_X+\Delta$ to the Galois closure of $Y/X$ might have negative coefficients.

\begin{proof}[Proof of Corollary \ref{c-localfund}]
For the germ of a klt singularity $(X,\Delta)$, by \cites{Xu14,Braun-local-pi_1} (see also \cite{TX-local-pi_1}), the fundamental group $\pi_1(x,X^{\rm sm})$ of the smooth locus $X^{\rm sm}$ of  is finite.

Let $f\colon (Y,y)\to (X,x)$ be the universal cover of $X^{\rm sm}$ and let $\Delta_Y=f^*\Delta$. Then we have $K_Y+\Delta_Y=f^*(K_X+\Delta)$, hence by Theorem \ref{t-finitedegformula} we get $\hvol(y,Y,\Delta_Y)=\deg(f)\cdot \hvol(x,X,\Delta)$. By \cite{LX-cubic-3fold}*{Theorem A.4}, we also have $\hvol(y,Y,\Delta_Y)\le n^n$ with equality if and only if $y\in X$ is smooth and $\Delta_Y=0$. It follows that 
\[
\deg(f)=\#|\pi_1(x,X^{\rm sm})|\le \frac{n^n}{\hvol(x,X,\Delta)}
\]
and the equality holds if and only if $\big(y\in (Y,\Delta_Y)\big)\cong (0\in\bC^n)$ (\'etale locally), i.e., $\Delta=0$ and $(x\in X)$ is \'etale locally isomorphic to $\bC^n/G$ where $G\cong \pi_1(x,X^{\rm sm})$ and the action of $G$ is fixed point free in codimension one.
\end{proof}

\begin{proof}[Proof of Theorem \ref{t-Cartier}]
By \cite{BJ20}*{Theorem D}, we have
\[
\hvol(x,X,\Delta)\ge \left(\frac{n}{n+1}\right)^n \cdot \delta(X,\Delta)^n\cdot \left(-(K_X+\Delta)\right)^n.
\]
Thus the result follows immediately from Corollary \ref{c-localfund}.
\end{proof}

\bibliography{ref}

\end{document}